\documentclass{amsart}%
\usepackage{amsfonts}
\usepackage{amsmath}
\usepackage{amssymb}
\usepackage{graphicx}
\usepackage{hyperref}%
\providecommand{\U}[1]{\protect\rule{.1in}{.1in}}
\newtheorem{theorem}{Theorem}
\theoremstyle{plain}
\newtheorem{acknowledgement}{Acknowledgement}

\newtheorem{corollary}{Corollary}

\newtheorem{example}{Example}

\newtheorem{lemma}{Lemma}

\newtheorem{problem}{Problem}
\newtheorem{proposition}{Proposition}
\newtheorem{remark}{Remark}

\numberwithin{equation}{section}
\begin{document}
\title[{\normalsize On Divisor Topology of Commutative Rings}]{{\normalsize On Divisor Topology of Commutative Rings}}
\author{U\u{g}ur Y\.{I}\u{g}\.{I}t}
\address{Department of Mathematics, Istanbul Medeniyet University, \"{U}sk\"{u}dar
34700, Istanbul, Turkey.}
\email{ugur.yigit@medeniyet.edu.tr}
\author{Suat Ko\c{c}}
\address{Department of Mathematics, Istanbul Medeniyet University, \"{U}sk\"{u}dar
34700, Istanbul, Turkey.}
\email{suat.koc@medeniyet.edu.tr}

\begin{abstract}
Let $R\ $be an integral domain and $R^{\#}$ the set of all nonzero nonunits of
$R.\ $For every elements $a,b\in R^{\#},$ we define $a\sim b$ if and only if
$aR=bR,$ that is, $a$ and $b$ are associated elements.\ Suppose that
$EC(R^{\#})$ is the set of all equivalence classes of $R^{\#}\ $according to
$\sim$.$\ $Let $U_{a}=\{[b]\in EC(R^{\#}):b\ $divides $a\}$ for every $a\in
R^{\#}.$\ Then we prove that the family $\{U_{a}\}_{a\in R^{\#}}$ becomes a
basis for a topology on $EC(R^{\#}).\ $This topology is called divisor
topology of $R\ $and denoted by $D(R).\ $We investigate the connections
between the algebraic properties of $R\ $and the topological properties
of$\ D(R)$. In particular, we investigate the seperation axioms on $D(R)$,
first and second countability axioms, connectivity and compactness on $D(R)$.
We prove that for atomic domains $R,\ $the divisor topology $D(R)\ $is a Baire
space. Also, we characterize valution domains $R$ in terms of nested
property of $D(R).$ In the last section, we introduce a new topological proof
of the infinitude of prime elements in a UFD and integers by using the
topology $D(R)$.

\end{abstract}
\maketitle

\section{Introduction}

In the recent history of commutative algebra, many algebraists have created
new topologies on algebraic structures to explore the connections between
general topology and commutative algebra. One of the important topologies is
due to Oscar Zariski. Let $R$ be a commutative ring with nonzero identity and
$Spec(R)$ be the set of all prime ideals of $R.\ $For any $a\in R,$
$D(a)=\{P\in Spec(R):a\notin P\}$.\ Then the family $\{D(a):a\in R\}$ is a
basis for a topology on $Spec(R)$ which is called the Zariski topology of
$R.\ $The Zariski topology allows tools from topology to be used to study
algebraic varieties, and so can be viewed as one of the starting points of
modern algebraic geometry. For more information about the Zariski topology, we
refer \cite{McMoSmi}, \cite{Rinaldi} and \cite{YilErTeKo} to the reader.
Another example of topologies over algebraic structures is the Golomb topology
which was first defined by Golomb on the set of positive integers
\cite{Golomb} and was generalized to the integral domains by Clark et al.
\cite{Clark}. Other topologies over algebraic structures, one may consult
\cite{Ansari}, \cite{Atiyah}, \cite{Ceken}, \cite{Ceken2} and \cite{Ceken3}.

In this paper, we focus only on integral domains. Let $R$ will always denote
such a domain and $R^{\#}$ will denote the set of all nonzero nonunits of
$R.\ $For any two elements $a,b\in R^{\#},\ $we say that $a$ and $b$ are
associated, which we denote this by $a\sim b,$ if whenever $aR=bR.$ This
equivalently means that $a$ and $b$ divide each other. In fact, $\sim$ is an
equivalence relation on $R^{\#}$ and we denote the set of all equivalence
classes on $R^{\#}$ by $EC(R^{\#})=\{\left[  a\right]  :a\in R^{\#}\},\ $where
$\left[  a\right]  $ denotes the equivalence class of $a$ according to $\sim$.
\ For any $a\in R^{\#},$ let
\[
U_{a}=\{\left[  b\right]  \in EC(R^{\#}):\ b\ \text{divides}\ a\}.
\]
We prove in Theorem \ref{tdiv} that the family $\left\{  U_{a}\right\}  _{a\in
R^{\#}}$ is a basis for a topology on $EC(R^{\#})$ which is called the
\textit{divisor topology} of the ring $R$ and denoted by $D(R).\ $If we take
$R$ as the ring of integers, then $D(%
\mathbb{Z}
)$ corresponds to the divisor topology on the integers in the sense of
\cite{Steen}. So our concept is a natural extension of the divisor topology on
the integers to the context of integral domains. The divisor topology
$D(R)\ $of the domain $R$ leads to exploring the connections between the
algebraic properties of the ring $R$ and the topological properties of
$D(R).$\ Our main aim in this paper is to study the topological properties of
$D(R),$ and to characterize some important classes of rings in terms of the
topological properties of $D(R).\ $Among the other results, we investigate the
separation axioms on the divisor topology $D(R)$ (See, Proposition \ref{pT0},
Proposition \ref{pclosure}).\ Also, we determine all irreducible elements of
the ring $R$ in terms of isolated points of $D(R)\ $(See, Proposition
\ref{piso}).\ A domain $R$ with the quotient field $K$ is said to be a
\textit{valuation domain} if for every $0\neq x\in K,\ $either $x\in R$ or
$x^{-1}\in R\ $\cite{Larsen}. 
 topological space $(X,\tau)$ is called a \textit{nested space} if $\tau$ is linearly ordered with respect to
$\subseteq$, that is, every two open sets of $X\ $are comparable
\cite{Richmond}.\ Assume that $(X,\tau)$ is a topological space with the basis
$\mathfrak{B}$.\ First, we prove in Lemma \ref{l1} that $X$ is a nested space
if and only if either $B_{1}\subseteq B_{2}\ $or $B_{2}\subseteq B_{1}\ $for
every $B_{1},B_{2}\in\mathfrak{B}$. In terms of Lemma \ref{l1}, we show that a
domain $R$ is a valuation domain if and only if $D(R)\ $is a nested
topological space. Also, we show that $D(R)$ is both Alexandrov space and
Baire space (See, Proposition \ref{Alexandrov} and Theorem \ref{baire}). Let
$X$ be a topological space and $x\in X.\ $Recall from \cite{Munkres} that
$X\ $is said to have a countable basis at the point $x$ if there is a
countable family $\left\{  N_{k}\right\}  _{k\in%
\mathbb{N}
}$ of neighborhoods of $x$ such that any neighborhood $U\ $of $x$ contains at
least one of the sets $N_{k}.\ $A space $X\ $that has a countable basis at
each of its points is said to \textit{satisfy the first countability axiom}.
Also, a topological space $X\ $is said to \textit{satisfy the second
countability axiom} or \textit{the second countable} if $X\ $has a countable
basis. In Proposition \ref{pfirst} and Proposition \ref{psecond}, we
investigate the countability axioms on the divisor topology $D(R).$ The last
chapter is dedicated to the study of Noetherian space and the divisor
topology. A commutative ring $R\ $not necessarily a domain is said to be a
\textit{Noetherian ring} if every ascending chain
\[
I_{1}\subseteq I_{2}\subseteq\cdots\subseteq I_{n}\subseteq\cdots
\]
of ideals of $R$ terminates. It is well known that a ring $R\ $is a Noetherian
ring if and only if its every ideal is finitely generated if and only if every
family $\mathcal{F}$ of ideals of $R$ has a maximal element (i.e. $R$ has
maximum condition on its ideals) \cite{Larsen}. The class of Noetherian rings
is very important in both commutative rings and algebraic geometry since many
theorems and results depend on the finite generation of ideals of the ring.
Recall from \cite{Hars} that a topological space $X$ is said to be a
\textit{Noetherian space} if it satisfies the descending chain condition for
closed subsets of $X:\ $for every descending chain
\[
K_{1}\supseteq K_{2}\supseteq\cdots\supseteq K_{n}\supseteq\cdots
\]
of closed sets in $X,\ $there exists $n_{0}\in%
\mathbb{N}
$ such that $K_{n_{0}}=K_{n+n_{0}}\ $for every $n\geq1.\ $It is clear that
$X\ $is a Noetherian topological space if and only if it satisfies the
ascending chain condition for open subsets of $X.$\ We prove in Theorem
\ref{tNotherian} that a topological space $X$ with the basis$\mathcal{\ }%
\mathfrak{B\ }$is a Noetherian space if and only if $X\ $satisfies minimum
condition on closed sets if and only if $X\ $satisfies maximum condition on
open sets if and only if any subfamily $\mathfrak{B}^{\prime}$ of
$\mathfrak{B}$ has a maximal element. In terms of this result, we prove that
$D(R)\ $is a Noetherian space if and only if $R\ $is a field (See, Theorem
\ref{tnoe}). In the last section, we use the topology $D(R)$ to introduce a
new topological proof of the infinitude of prime elements in a UFD and
integers (See, Theorem \ref{infty}).

\section{\bigskip Properties of the divisor topology $D(R)$}

Throughout the study, all rings under consideration are assumed to be integral
domain. Let $R$ be an integral domain and $R^{\#}$ the set of all nonzero
nonunits in $R$.$\ $Define the relation $\sim$ on $R^{\#}$ as follows: for
every $a,b\in R^{\#}$, $a\sim b$ if and only if $aR=bR.\ $It is clear that
$\sim$ is an equivalence relation on $R^{\#},\ $and we denote the set of all
equivalence classes on $R^{\#}$ by $EC(R^{\#})=\{\left[  a\right]  :a\in
R^{\#}\},\ $where $\left[  a\right]  $ denotes the equivalence class of $a$
according to $\sim.\ $Also for any $a\in R^{\#},$ we define $U_{a}=\{\left[
b\right]  \in EC(R^{\#}):\ b\ $divides $a\}.$ In order to show that
the$\ $family $\left\{  U_{a}\right\}  _{a\in R^{\#}}$ is a basis for a
topology on $EC(R^{\#})$, we need the following result.

An integral domain $R\ $is called a \textit{GCD-domain} if for any $0\neq
a,b\in R$, $\gcd(a,b)$\ exists. Evidently, $R\ $is a GCD-domain if and
only if $\operatorname{lcm}(a,b)$\ exists for any $0\neq a,b\in R\ $%
\cite{Anderson}. Note that every UFD and Bez\'{o}ut domains are also a
GCD-domain but the converse is generally not true. For instance, the ring of
all algebraic integers $\mathcal{O}$ is a Bez\'{o}ut domain, so is a
GCD-domain which is not a UFD. Also, the polynomial ring $%
\mathbb{Z}
\lbrack X]$ is a GCD-domain which is not a Bez\'{o}ut domain.

\begin{proposition}
\label{p1}Assume that $R$ is an integral domain. The following statements are satisfied.

(i) For any $a\in R^{\#},$ we have$\ [a]\in U_{a}.$

(ii) For any $a,b\in R^{\#},\ a$ divides $b$ if and only if $U_{a}\subseteq
U_{b}.$

(iii) $%
{\textstyle\bigcup\limits_{a\in R^{\#}}}
U_{a}=EC(R^{\#}).$

(iv)\ If $[x]\in U_{a}\cap U_{b}$ for some $a,b\in R^{\#},\ $then $[x]\in
U_{x}\subseteq U_{a}\cap U_{b}.$

(v)\ If $R\ $is a GCD-domain, then $U_{a}\cap U_{b}=\emptyset$ or $U_{a}\cap
U_{b}=U_{\gcd(a,b)}$ for all $a,b\in R^{\#}$.

(vi) Assume that $R$ is a GCD-domain. If $U_{a}\subseteq U_{c}$ and $U_{b}\subseteq
U_{c}\ $for some $a,b,c\in R^{\#},$ then $U_{\operatorname{lcm}(a,b)}\subseteq
U_{c}. $
\end{proposition}

\begin{proof}
$(i),(ii),(iii):$ Clear.

$(iv):\ $Follows from $(i)\ $and $(ii).$

$(v):\ $The intersection of $U_{a}$ and $U_{b}$ can be empty (actually, this
occurs in the case that $a$ and $b$ are relatively prime). Otherwise, the
inclusion $U_{\gcd(a,b)}\subseteq U_{a}\cap U_{b}$ follows from $(ii)\ $since
$\gcd(a,b)$ divides $a$ and $b.\ $Let $[x]\in U_{a}\cap U_{b}$. Then
$x\ $divides $a$ and $x$ divides $b$. This implies that $x\ $divides
$\gcd(a,b).\ $Then by $(i)\ $and $(ii)$,\ $[x]\in U_{x}\subseteq U_{\gcd
(a,b)}.\ $This completes the proof.

$(vi):\ $Suppose that $U_{a}\subseteq U_{c}$ and $U_{b}\subseteq U_{c}\ $for
some $a,b,c\in R^{\#}.\ $Then by $(ii),\ a$ divides $c$ and $b$ divides
$c.\ $This implies that $\operatorname{lcm}(a,b)\ $divides $c$.\ Then by
$(ii),\ $we have $U_{\operatorname{lcm}(a,b)}\subseteq U_{c}.$
\end{proof}

\begin{theorem}
\label{tdiv}Assume that $R\ $is an integral domain. The family $\left\{
U_{a}\right\}  _{a\in R^{\#}}$ is a basis for a topology on $EC(R^{\#}).$ In
this case, this topology is called the divisor topology of $R$ and denoted by
$D(R).$
\end{theorem}

\begin{proof}
Follows from Proposition \ref{p1} $(iii)$ and $(iv)$.
\end{proof}

\begin{example}
Consider the ring of integers $%
\mathbb{Z}
.\ $Then $%
\mathbb{Z}
^{\#}=\{\pm n:n\in%
\mathbb{N}
$ and $n\geq2\}$ and $EC(%
\mathbb{Z}
^{\#})=\{[n]:n\in%
\mathbb{N}
\ $and $\geq2\}.$\ For any $n\geq2,\ $we have $U_{n}=\{[x]\in EC(%
\mathbb{Z}
^{\#}):$ $x$ is a positive divisor of $n\}.\ $In fact, this topology is the
divisor topology on the integers (See, \cite{Steen}).
\end{example}

Another nontrivial example is given in Example \ref{Ex2}.

\begin{remark}
\label{remark}Assume that $R$ is an integral domain. Then $R$ is a field if and only
if $R^{\#}\ $is empty set if and only if $EC(R^{\#})$ is empty set. In this
case, $D(R)\ $is empty space. So from now on, we assume that all integral
domains are not field otherwise stated.
\end{remark}

Recall from \cite{Munkres} that a topological space $X$ is called a
$T_{0}$\textit{-space} if for every two distinct points $x,y\in X,\ $there
exists an open set $O$ containing $x$ but not containing $y$ (or containing
$y$ but not containing $x).\ $

\begin{proposition}
\label{pT0}Assume that $R$ is an integral domain. Then $D(R)$ is a $T_{0}$-space.
\end{proposition}

\begin{proof}
Let $[a]\neq\lbrack b]$ in $EC(R^{\#}).\ $Then $a$ does not divide $b$ or $b$
does not divide $a.\ $Otherwise, we would have $[a]=[b].\ $We assume, without losing generality, that $a$ does not divide $b$. Then by Proposition
\ref{p1} $(i)\ $and $(ii),\ $we have $[a]\notin U_{b}$ and $[b]\in U_{b}.$
\end{proof}

Let $X\ $be a topological space and $x\in X.\ $Recall from \cite{Munkres} that
$x$ is said to be an \textit{isolated point} if $\{x\}\ $is open in
$X.\ $

\begin{proposition}
\label{piso}Assume that $R$ is an integral domain. Then $[a]\in EC(R^{\#})\ $is an
isolated point if and only if $a$ is an irreducible element of $R.\ $Especially, if $R\ $is a GCD-domain, then $[a]\in EC(R^{\#})\ $is an isolated
point if and only if $a$ is a prime element of $R.$
\end{proposition}

\begin{proof}
$\left(  \Rightarrow\right)  :\ $Suppose that $[a]\in EC(R^{\#})$ is an
isolated point. Then there exists an open set $O=%
{\textstyle\bigcup\limits_{x\in\Delta}}
U_{x},\ $where $\Delta$ is an arbitrary subset of $R^{\#},\ $such that
\[
\left\{  [a]\right\}  =O=%
{\textstyle\bigcup\limits_{x\in\Delta}}
U_{x}.\
\]
Since $\left\{  [a]\right\}  \ $is singleton, we have $\left\vert
\Delta\right\vert =1,$ that is, $\left\{  [a]\right\}  =U_{x}\ $for some $x\in
R^{\#}.\ $Since $[x]\in U_{x},\ $we have $[a]=[x]\ $which implies that
$a,x\ $are associated elements. This implies that $U_{a}=U_{x}.\ $Choose a
divisor $t$ of $a$ which is not unit. Then $[t]\in U_{a}=\left\{  [a]\right\}
.\ $Thus we conclude that $[t]=[a],\ $that is, $t\ $and $a$ are associated
elements. Hence, $a$ is an irreducible element.

$\left(  \Leftarrow\right)  :$ Now, assume that $a$ is an irreducible element
of $R.\ $Let $[x]\in U_{a}\ $for some $x\in R^{\#}.\ $Then $x\ $divides
$a.\ $As $a$ is an irreducible element and $x$ is not unit, we have $x$ and
$a$ are associated. This implies that $[a]=[x],\ $that is, $U_{a}=\{[a]\}$ is
an open set. Hence, $[a]$ is an isolated point.

Since in a GCD-domain prime elements and irreducible elements are equal, the
rest is clear.
\end{proof}

Assume that $R\ $is an integral domain and $K\ $its quotient field. Recall from
\cite{Larsen} that $R\ $is said to be a \textit{valuation domain} if for each
$x\in K,\ $either $x\in R$ or $x^{-1}\in R.\ $By \cite[Proposition
5.2]{Larsen}, we know that an integral domain $R\ $is a valuation domain if
and only if for each $a,b\in R,\ $either $aR\subseteq bR$ or $bR\subseteq
aR,\ $equivalently, $a$ divides $b$ or $b$ divides $a.\ $

A topological space $(X,\tau)$ is called a \textit{nested space} if $\tau$
is linearly ordered with respect to $\subseteq$ \cite{Richmond}.$\ $It is
clear that $(X,\tau)$ is a nested space if and only if every two open sets of
$X\ $are comparable, that is, $O_{1}\subseteq O_{2}\ $or $O_{2}\subseteq
O_{1}$ for every $O_{1},O_{2}\in\tau.\ $

\begin{lemma}
\label{l1}Assume that $(X,\tau)$ is a topological space with basis $\mathfrak{B}%
$.\ Then $X\ $is a nested space if and only if either $B_{1}\subseteq B_{2}%
\ $or $B_{2}\subseteq B_{1}\ $for every $B_{1},B_{2}\in\mathfrak{B}$.
\end{lemma}

\begin{proof}
$\left(  \Rightarrow\right)  :\ $It is obvious.\newline$\left(  \Leftarrow
\right)  :\ $Let $U,V\in\tau$. Then we can write $U=%
{\textstyle\bigcup\limits_{u\in\Delta_{1}}}
B_{u}$ and $V=%
{\textstyle\bigcup\limits_{v\in\Delta_{2}}}
B_{v}$ where $B_{u},B_{v}\in\mathfrak{B}$ and $\Delta_{1},\Delta_{2}$ are
arbitrary index sets. Assume that $V\nsubseteq U.\ $Then for some $v\in
\Delta_{2},\ B_{v}\nsubseteq U.\ $This implies that $B_{v}\nsubseteq B_{u}%
\ $for each $u\in\Delta_{1}.\ $By the assumption, we have $B_{u}\subseteq
B_{v}$ for each $u\in\Delta_{1}.\ $Thus, we conclude that $%
{\textstyle\bigcup\limits_{u\in\Delta_{1}}}
B_{u}=U\subseteq B_{v}\subseteq V.\ $Hence, $X\ $is a nested space.
\end{proof}

\begin{theorem}
\label{tnested}Assume that $R\ $is an integral domain. Then $R\ $is a valuation domain
if and only if $D(R)\ $is a nested topological space.
\end{theorem}

\begin{proof}
Assume that $R\ $is a valuation domain. We will show that $D(R)$ is a
nested space. Choose two basis element $U_{a}$ and $U_{b}\ $for some $a,b\in
R^{\#}.\ $Because $R\ $is a valuation domain, by \cite[Proposition 5.2]{Larsen},
we have either $a$ divides $b$ or $b$ divides $a.\ $Then by Proposition
\ref{p1}, $U_{a}\subseteq U_{b}\ $or $U_{b}\subseteq U_{a}.\ $By Lemma
\ref{l1}, we have that $D(R)\ $is a nested space. For the converse, assume
that $D(R)\ $is a nested space. Let $a,b$ be two nonzero nonunit elements of
$R.\ $Then by assumption, $U_{a}\subseteq U_{b}$ or $U_{b}\subseteq U_{a}%
.\ $By Proposition \ref{p1}, we have either $a$ divides $b$\ or $b$ divides
$a.\ $Then by \cite[Proposition 5.2]{Larsen}, $R\ $is a valuation domain.
\end{proof}

\begin{example}
\label{Ex2} Consider the formal power series ring $R=k[[x]]$ over an arbitrary
field $k.$\ We know that $R=k[[x]]\ $is a valuation domain. Also for every
nonzero nonunit $f\in R$ has the form $f=x^{k}g$ for some positive integer
$k\geq1$ and a unit $g$ of $R.\ $Then we have $fR=x^{k}R.\ $Then by
Proposition \ref{p1}, $U_{f}=U_{x^{k}}.\ $Since $x^{k}$ divides $x^{m}\ $for
every positive integer $1\leq k\leq m,\ $again by Proposition \ref{p1},
$U_{x^{k}}\subseteq U_{x^{m}}.\ $Thus the basis $\mathfrak{B}=\{U_{x^{k}%
}\}_{k\in%
\mathbb{N}
}$ is linearly ordered. Then by Lemma \ref{l1}, $D(k[[x]])$ is a nested space.
\end{example}

\begin{lemma}
\label{l2}Let $R\ $be an integral domain and $[a]\in EC(R^{\#}).\ $For every
open set $O$ containing $[a]$ also contains $U_{a}.\ $Thus, $U_{a}\ $is the
smallest open set containing $[a].$
\end{lemma}

\begin{proof}
Suppose that $O\ $is an open set containing $[a].\ $Then we can write $O=%
{\textstyle\bigcup\limits_{x\in\Delta}}
U_{x},\ $where $\Delta$ is an arbitrary subset of $R^{\#}.$\ Since $[a]\in
O,\ $then there exists $x\in\Delta$ such that $[a]\in U_{x}.\ $This gives
$a\ $divides $x,\ $by Proposition \ref{p1}, $U_{a}\subseteq U_{x}\subseteq
O.\ $Thus, $U_{a}\ $is the smallest open set containing $[a].$
\end{proof}

A topological space $X\ $is called an \textit{Alexandrov space} if the
intersection of arbitrary family of open sets is also open \cite{Alex}. This
is equivalent to every point $x$ of $X\ $has a minimal neighborhood. The
following result is an immediate consequence of Lemma \ref{l2}.

\begin{proposition}
\label{Alexandrov}Suppose that $R\ $is an integral domain. Then $D(R)\ $is an
Alexandrov space.
\end{proposition}

Recall from \cite{Munkres} that a topological space $X\ $is called a
$T_{1}$\textit{-space} if for every two distinct points $x,y\in X,\ $then
there exist two open sets $U,V$ such that $x\in U,y\notin U$ and $y\in
V,x\notin V.\ $Note that $X\ $is a $T_{1}$-space if and only if every
singleton $\{x\}$ is closed in $X$. Also, a topological space $X\ $is said to
be a $T_{2}$\textit{-space} (or \textit{Hausdorff}) if for every two distinct
points $x,y\in X,\ $then there exist two disjoint open sets $U,V$ such that
$x\in U$ and $y\in V$ \cite{Munkres}. It is clear that every Hausdorff space
is also $T_{1}$ and the converse is generally not true (See, \cite{Munkres}).

\begin{proposition}
\label{pclosure}Let $R$ be an integral domain and $[a]\in EC(R^{\#}). $\ Then
\[
\ \overline{\left\{  [a]\right\}  }=\{[b]\in EC(R^{\#}):a \ \text{divides}
\ b\}.
\]
Thus, $D(R)$ is not a $T_{1}$-space and not Hausdorff.
\end{proposition}

\begin{proof}
Let $[b]\in\overline{\{[a]\}}\ $for some $b\in R^{\#}.\ $For every open set
$O$ containing $[b]$ also contains $[a].\ $Choose $O=U_{b}\ $and note that
$[a]\in U_{b}.\ $Then we have $a$ divides $b.\ $For the converse, assume that
$a$ divides $b$ for some $b\in R^{\#}.$\ Choose an open set $O$ containing
$[b].\ $Then by Lemma \ref{l2}, $U_{b}\subseteq O.\ $Since $a$ divides
$b,\ $by Proposition \ref{p1}, we have $[a]\in U_{a}\subseteq U_{b}\subseteq
O.\ $Hence, we obtain $[b]\in\overline{\{[a]\}}.\ $Consequently, we have
$\overline{\{[a]\}}=\{[b]\in EC(R^{\#}) :a$ divides $b\}.\ $Since $R\ $is not
a field, there exists a nonzero nonunit element $a\in R.\ $This gives
$[a^{2}]\in\overline{\{[a]\}}$ and $[a^{2}]\neq\lbrack a]\ $which implies that
$\ \overline{\{[a]\}}\neq\{[a]\}.\ $Hence, $D(R)$ is not a $T_{1}$-space and
not Hausdorff.
\end{proof}

Recall from \cite{Cohn} that an integral domain $R\ $is called an
\textit{atomic domain} if every nonzero nonunit element $x\in R$ can be
written as a finite product of irreducible elements of $R.\ $Note that in
atomic domain, decomposition of an element into irreducibles need not be
unique, in other words, irreducible element may not be prime. Thus every UFD
is an atomic domain but the converse is not true in general. Also, by
\cite[Proposition 1.1]{Cohn}, an integral domain is an atomic domain if and
only if it satisfies the maximum condition on principal ideals, namely, every
family of principal ideals has a maximal element with respect to
$\subseteq.\ $Thus every Noetherian domain is an example of atomic domain. For
an integral domain $R,\ $we denote the set of all irreducible elements by
$Irr(R)$ and also we denote the set of all equivalence classes of irreducible
elements according to $\sim$ by
\[
EC(Irr(R))=\{[a]:a\in Irr(R)\}.
\]

\begin{theorem}
\label{tdense}Let $R\ $be an atomic domain. $EC(Irr(R))$ is dense in
$D(R)$.\ Also, every open dense set in $D(R)\ $contains $EC(Irr(R)).$
\end{theorem}

\begin{proof}
First, we will show that $\overline{EC(Irr(R))}=EC(R^{\#}).\ $Let $a\in R$ be
an arbitrary nonzero nonunit element. Since $R\ $is atomic, we can write
$a=q_{1}q_{2}\cdots q_{k}\ $for some irreducible elements $q_{1},q_{2}%
,\ldots,q_{k}$ of $R.\ $Then by Proposition \ref{pclosure}, $[a]\in
\overline{\{[q_{1}]\}}\subseteq\overline{EC(Irr(R))}.\ $This implies that
$\overline{EC(Irr(R))}=EC(R^{\#}),$ that is, $EC(Irr(R))$ is dense in
$D(R).$\ Now, take an open dense set $O$ in $D(R).$\ Then $\overline
{O}=EC(R^{\#}).\ $Let $q\in Irr(R).\ $Then $[q]\in$ $EC(R^{\#})=\overline
{O}.\ $Thus for every open set $V\ $containing $[q]$ intersects with
$O.\ $Choose $V=U_{q}.\ $Then by Proposition \ref{piso}, $V=U_{q}=\{[q]\}$.
Thus, we have $U_{q}\cap O\neq\emptyset$ which implies that $[q]\in
O$.\ Hence, we have $EC(Irr(R))\subseteq O.\ $
\end{proof}

Recall from \cite{Munkres} that a topological space $X\ $is called a
\textit{Baire space} if for every countable family $\{A_{n}\}_{n\in%
\mathbb{N}
}$ of closed sets having empty interior in $X$, then $%
{\textstyle\bigcup\limits_{n\in\mathbb{N}}}
A_{n}$ has also empty interior in $X.$ By Lemma \cite[Lemma 48.1]{Munkres},
$X\ $is a Baire space if and only if every countable family $\{O_{n}\}_{n\in%
\mathbb{N}
}$ of open sets in $X,\ $each of which is dense in $X,\ $their intersection $%
{\textstyle\bigcap\limits_{n\in\mathbb{N}}}
O_{n}\ $is also dense in $X.\ $

\begin{theorem}
\label{baire}Let $R\ $be an atomic domain. Then $D(R)\ $is a Baire space.
\end{theorem}

\begin{proof}
Let $R$ be an atomic domain and $\{O_{n}\}_{n\in%
\mathbb{N}
}$ be a family of open sets in $X$ such that every $O_{n}\ $is dense in
$D(R).\ $Then by Theorem \ref{tdense}, for every $n\in%
\mathbb{N}
,\ $we have $O_{n}\supseteq EC(Irr(R)).\ $Which implies that $%
{\textstyle\bigcap\limits_{n\in\mathbb{N} }}
O_{n}\supseteq EC(Irr(R)).\ $Again by Theorem \ref{tdense}, $\overline{%
{\textstyle\bigcap\limits_{n\in\mathbb{N} }}
O_{n}}\supseteq\overline{EC(Irr(R))}=EC(R^{\#}),\ $that is, $\overline{%
{\textstyle\bigcap\limits_{n\in\mathbb{N} }}
O_{n}}=EC(R^{\#}).\ $Since $%
{\textstyle\bigcap\limits_{n\in\mathbb{N} }}
O_{n}$ is dense in $D(R),\ $by \cite[Lemma 48.1]{Munkres}, $D(R)\ $is a Baire space.
\end{proof}

Let $X$ be a topological space and $x\in X.\ $Recall from \cite{Munkres} that
$X\ $is said to \textit{have a countable basis at the point }$x$ if there is a
countable family $\left\{  N_{k}\right\}  _{k\in%
\mathbb{N}
}$ of neighborhoods of $x$ such that any neighborhood $U\ $of $x$ contains at
least one of the sets $N_{k}.\ $A space $X\ $that has\ countable basis at each
of its points is said to \textit{satisfy the first countability axiom}.

\begin{proposition}
\label{pfirst}Assume that $R\ $is an integral domain. Then $D(R)\ $satisfies the first
countability axiom.
\end{proposition}

\begin{proof}
Let $R\ $be an integral domain and $[a]\in EC(R^{\#}).$ By Lemma \ref{l2},
$U_{a}\ $is the smallest neighborhood of $[a].\ $Thus, $D(R)\ $satisfies the
first countability axiom.
\end{proof}

Recall from \cite{Munkres} that a topological space $X\ $is said to
\textit{satisfy second countability axiom} or \textit{second countable} if
$X\ $has a countable basis.

\begin{proposition}
\label{psecond}Assume that $R\ $is an integral domain and $EC(R^{\#})$ has countable
elements. Then $D(R)\ $is second countable. In particular, if $R\ $is a
countable integral domain, then $D(R)\ $is second countable.
\end{proposition}

\begin{proof}
Follows from Theorem \ref{tdiv} and the assumption $EC(R^{\#})$ has countable
elements. The rest is clear.
\end{proof}

\begin{remark}
If$\ R\ $is a countable integral domain, then $EC(R^{\#})$ has countable
elements. However, the converse is generally not true. Consider the formal
power series ring $%
\mathbb{R}
\lbrack\lbrack x]]$ over the real numbers $%
\mathbb{R}
.\ $Then, clearly $%
\mathbb{R}
\lbrack\lbrack x]]$ is not countable but $EC(%
\mathbb{R}
\lbrack\lbrack x]]^{\#})=\{\left[  x^{k}\right]  :k\in%
\mathbb{N}
\}$ has countable elements.
\end{remark}

\begin{problem}
Is $EC(R^{\#})$ countable provided that $D(R)$ is second countable?
\end{problem}

Recall from \cite{Steen} that a topological space $X\ $is said to be an
\textit{ultraconnected space} if the intersection of two nonempty closed set
is nonempty. Also, a topological space $X\ $is called a \textit{normal
space} if for every disjoint closed sets $F\ $and $K,\ $there exist open
disjoint sets $U,V\ $such that $F\subseteq U$ and $K\subseteq V\ $%
\cite{Munkres}.

\begin{proposition}
\label{pult}Assume that $R\ $is an integral domain. Then $D(R)\ $is an ultraconnected
space. Hence, $D(R)\ $is a normal space.
\end{proposition}

\begin{proof}
Let $K_{1},K_{2}$ be two nonempty closed sets of $D(R).\ $Then there exist
$[a]\in K_{1}$ and $[b]\in K_{2}$ for some $a,b\in R^{\#}.\ $Then by
Proposition \ref{pclosure}, we have $[ab]\in\overline{\{[a]\}}\cap
\overline{\{[b]\}}\subseteq K_{1}\cap K_{2}.\ $Hence, $K_{1}\cap K_{2}$ is
nonempty, that is, $D(R)\ $is an ultraconnected space. The rest is clear.
\end{proof}

Recall from \cite{Munkres} that a topological space $X$ is called a connected space if it cannot be written as the union of disjoint nonempty open subsets of $X$. Let $x,y\in X$. A path from $x$ to $y$ in $X$ is a continuous map $f:[a,b]\rightarrow X$ such that $f(a)=x$ and $f(b)=y$. A topological space $X$ is said to be a path connected space if there exists a path in $X$ for every pair of elements of $X$. 

\begin{corollary}
Assume that $R\ $is an integral domain. Then $D(R)\ $is a path connected space, so a
connected space.
\end{corollary}

\begin{proof}
Note that every ultraconnected space is a path connected space and every path
connected space is a connected space. So the claim follows from Proposition
\ref{pult}.
\end{proof}

Recall from \cite{Munkres} that a completely normal space is a topological
space X such that every subspace of $X$ is a normal space with the subspace
topology. $X$ is a completely normal space if and only if for every two separated sets
in $X$ (that is, two sets $A$ and $B$ are said to be \textit{seperated sets} if
$\overline{A}\cap B=\emptyset$ and $A\cap\overline{B}=\emptyset$), there exist
disjoint open neighbourhoods containing them.

\begin{proposition}
Assume that $R\ $is an integral domain such that $\left\vert EC(R^{\#})\right\vert
\geq3$. Then $D(R)\ $is not a completely normal space.
\end{proposition}

\begin{proof}
Let $a,b,c\in R$ be nonassociated irreducible elements. Then by Proposition
\ref{pclosure}, $\{[ab]\}$ and $\{[ac]\}$ are separated sets for $[ab],[ac]\in
EC(R^{\#})$. However, they have no disjoint open neighborhoods, since their
smallest open neighborhoods intersect $U_{ab}\cap U_{ac}\supseteq U_{a}$.
\end{proof}

Recall from \cite{Munkres} that a topological space $X$, in which one-point
sets are closed in $S$, is a regular space if given any closed set $F$ and any
point $x$ which is not an element of $F$, there exist disjoint open
neighborhood $U$ of $x$ and a neighborhood $V$ of $F$.

\begin{proposition}
\label{propreg}Assume that $R\ $is an integral domain. Then $D(R)\ $is not a regular space.
\end{proposition}

\begin{proof}
One point sets are not closed in $D(R)$ by Proposition \ref{pclosure}, then
the result follows.
\end{proof}

\begin{remark}
Since $D(R)$ is a $T_{0}$-space, if $D(R)$ were a regular space, $D(R)$ would
also be a $T_{3}$-space. Every $T_{3}$-space is also a Hausdorff space, $D(R)$
would be a Hausdorff space. However, we showed in Proposition \ref{pclosure}
that $D(R)$ is not a Hausdorff space. Then, $D(R)$ cannot be a regular space.
This is another proof of the Proposition \ref{propreg}.
\end{remark}

Note from \cite{Munkres} that a topological space $X$ is called a \textit{compact space} if every open cover of
$X$ has a finite subcover. That is, $X$ is called compact if for every
collection $\{U_{i}\}_{i\in I}$ of open subsets of $X$ such that
$\displaystyle X=\bigcup_{i\in I}U_{i}$, there is a finite subcollection
$J\subset I$ such that $\displaystyle X=\bigcup_{i\in J}U_{i}$.

Let $X$ be a set and $\textstyle{\mathcal{A}}$ a nonempty collection of subsets of
$X$. $\textstyle{\mathcal{A}}$ is said to have the finite intersection
property if every nonempty finite subfamily of ${\mathcal{A}}$ has nonempty intersection.

\begin{theorem}
\cite[Teorem 26.9.]{Munkres} Let $X$ be a topological space. Then $X$ is a
compact space if and only if for every collection $\textstyle {\mathcal{A}}$
of closed sets in $X$ having the finite intersection property, the
intersection $\displaystyle \bigcap_{A\in\mathcal{A}}A$ of all the elements of
$\textstyle {\mathcal{A}}$ is nonempty.
\end{theorem}

\begin{proposition} Assume that $R\ $is an integral domain. Then $D(R)\ $is not a compact space.
\end{proposition}

\begin{proof}
Let $\mathcal{A}=\bigl\{\overline{\{[a]\}}\ |\ [a]\in EC(R^{\#})\bigr\}$ be
the set of closure of points in $EC(R^{\#})$. Then by Proposition
\ref{pclosure},$\mathcal{A}$ has the finite intersection property. Now we will
show that $\bigcap_{\overline{\{[a]\}}\in\mathcal{A}}\overline{\{[a]\}}$ is
empty. Suppose that $\bigcap_{\overline{\{[a]\}}\in\mathcal{A}}\overline
{\{[a]\}}\neq\emptyset$. Choose $\left[  x\right]  \in\bigcap_{\overline
{\{[a]\}}\in\mathcal{A}}\overline{\{[a]\}}$. Then $x$ is nonzero nonunit and
$\left[  x\right]  \in\overline{\{[x^{2}]\}}$. This gives $x=x^{2}y$ for some
$y\in R$ which implies that $xy=1$. Thus $x$ is a unit of $R\ $which is a
contradiction. By the previous theorem, then $D(R)\ $cannot be a compact space.
\end{proof}

\section{Noetherian Topological Space and Divisor topology}

This section is dedicated to the study of Noetherian topological space and the
Noetherian property of the divisor topology. Recall from \cite{Hars} that a
topological space $X$ is called a \textit{Noetherian space} if it
satisfies the descending chain condition for closed subsets of $X:\ $for every
descending chain
\[
K_{1}\supseteq K_{2}\supseteq\cdots\supseteq K_{n}\supseteq\cdots
\]
of closed sets in $X,\ $there exists $n_{0}\in%
\mathbb{N}
$ such that $K_{n_{0}}=K_{n+n_{0}}\ $for every $n\geq1.\ $Evidently, $X\ $is a Noetherian topological space if and only if it satisfies the
ascending chain condition for open subsets of $X:\ $for every ascending chain
\[
O_{1}\subseteq O_{2}\subseteq\cdots\subseteq O_{n}\subseteq\cdots
\]
of open sets in$\ X,\ $there exists $n_{0}\in%
\mathbb{N}
$ such that $O_{n_{0}}=O_{n+n_{0}}\ $for every $n\geq1.\ $A topological space
$X\ $is said to \textit{satisfy maximum condition (minimum condition) on open
sets (closed sets)} if each family of open (closed)\ subsets $\{O_{\alpha
}\}_{\alpha\in\Delta}$\ $\left(  \{F_{\alpha}\}_{\alpha\in\Delta}\right)  $ of
$X\ $has a maximal element (minimal element) with respect to $\subseteq.\ $

\begin{theorem}
\label{tNotherian}Let $(X,\tau)$ be a topological space with the basis
$\mathfrak{B}$. Then the following statements are equivalent.

(i)\ $X$ is a Noetherian topological space.

(ii) $X\ $satisfies minimum condition on closed sets.

(iii)\ $X\ $satisfies maximum condition on open sets.

(iv) Any subfamily $\mathfrak{B}^{\prime}$ of $\mathfrak{B}$ has a maximal element.
\end{theorem}

\begin{proof}
$(i)\Rightarrow(ii):\ $Let $X\ $be a Noetherian topological space and
$\mathfrak{F}=\{F_{\alpha}\}_{\alpha\in\Delta}$ be a family of closed sets of
$X.\ $Assume that $\mathfrak{F}$ has not minimal element. Since $F_{\alpha
_{1}}$ is not minimal element, there exists $\alpha_{2}\in\Delta$ such that
$F_{\alpha_{1}}\supsetneq F_{\alpha_{2}}.\ $Since $F_{\alpha_{2}}$ is not
minimal element, there exists $\alpha_{3}\in\Delta$ such that $F_{\alpha_{2}%
}\supsetneq F_{\alpha_{3}}.\ $If we continue in this manner, we can obtain the
following infinite descending chain of closed sets
\[
F_{\alpha_{1}}\supsetneq F_{\alpha_{2}}\supsetneq F_{\alpha_{3}}\supsetneq
F_{\alpha_{4}}\supsetneq\cdots\supsetneq F_{\alpha_{n}}\supsetneq\cdots.
\]
Thus, $X\ $can not be Noetherian space, a contradiction. Hence, $X\ $satisfies
minimum condition on closed sets.

$(ii)\Rightarrow(iii):\ $Assume that $X\ $satisfies minimum condition on
closed sets. Let $\mathfrak{O}=\{O_{\alpha}\}_{\alpha\in\Delta}$ be a family
of open sets of $X.\ $Consider the family of closed sets $\mathfrak{F=}%
\{X-O_{\alpha}\}_{\alpha\in\Delta}$ of $X.\ $By the assumption (ii), there
exists a minimal element $X-O_{\alpha_{j}}$ of $\mathfrak{F}$ for some
$\alpha_{j}\in\Delta.\ $Let $O_{\alpha_{j}}\subseteq O_{\alpha_{i}}\ $for some
$\alpha_{i}\in\Delta.\ $Then we have $X-O_{\alpha_{i}}\subseteq X-O_{\alpha
_{j}}.$\ Since $X-O_{\alpha_{j}}$ is a minimal element of $\mathfrak{F}$, we
conclude that $X-O_{\alpha_{j}}=X-O_{\alpha_{i}}$ which implies that
$O_{\alpha_{j}}=O_{\alpha_{i}}.\ $Thus, $O_{\alpha_{j}}$ is a maximal element
of $\mathfrak{O}.$

$(iii)\Rightarrow(iv):$ It is clear since $\mathfrak{B}^{\prime}$ consists of
open sets in $X.\ $

$(iv)\Rightarrow(i):\ $Let
\[
O_{1}\subseteq O_{2}\subseteq\cdots\subseteq O_{n}\subseteq\cdots
\]
be an ascending chain of open sets in$\ X.\ $Then for each $n\in%
\mathbb{N}
,\ $we can write $O_{n}=%
{\textstyle\bigcup\limits_{\alpha\in\Delta_{n}}}
B_{\alpha}$ for some $B_{\alpha}\in\mathfrak{B}.\ $Let $\mathfrak{B}%
_{n}^{\prime}=\{B_{\alpha}:\ \alpha\in\Delta_{n}\}$ and $\mathfrak{B}^{\prime
}$ be the family of all finite unions of the subfamily $%
{\textstyle\bigcup\limits_{n\in\mathbb{N}}}
\mathfrak{B}_{n}^{\prime}\subseteq\mathfrak{B}.\ $Then by assumption, there
exists a maximal element $U\ $of $\mathfrak{B}^{\prime}.\ $Then there exist
$B_{\alpha_{1}}\in\mathfrak{B}_{k_{1}}^{\prime},B_{\alpha_{2}}\in
\mathfrak{B}_{k_{2}}^{\prime},\ldots,B_{\alpha_{m}}\in\mathfrak{B}_{k_{m}%
}^{\prime}\ $such that $U=B_{\alpha_{1}}\cup B_{\alpha_{2}}\cup\cdots\cup
B_{\alpha_{m}}.\ $Then for \ any $B_{\alpha}\in\mathfrak{B}^{\prime},\ $since
$U\subseteq U\cup B_{\alpha},$ we have $U=U\cup B_{\alpha}$ which implies that
$B_{\alpha}\subseteq U.\ $This gives $O_{n}\subseteq U$ for every $n\in%
\mathbb{N}
.\ $In particular,
\[
O_{k_{1}}\cup O_{k_{2}}\cup\cdots\cup O_{k_{m}}\subseteq U\subseteq O_{k_{1}%
}\cup O_{k_{2}}\cup\cdots\cup O_{k_{m}}.\
\]
Without loss of generality, we may assume that $k_{1}\leq k_{2}\leq\cdots\leq
k_{m}.\ $Then we have
\[
O_{k_{1}}\cup O_{k_{2}}\cup\cdots\cup O_{k_{m}}=O_{k_{m}}=U
\]
and also $O_{n}\subseteq O_{k_{m}}$ for every $n\in%
\mathbb{N}
.\ $This implies that $O_{k_{m}}=O_{n}\ $for every $n\geq k_{m}.\ $Thus, $X$
is a Noetherian topological space.
\end{proof}

Up to now, we assume that $R$ is not a field for the sake of $D(R)$ is empty
space. Here, we drop the condition $R$ is not a field for the following
result. Now, we are ready to determine the conditions under which $D(R)\ $is a
Noetherian space.

\begin{theorem}
\label{tnoe}Assume that $R\ $is an integral domain. The following statements are equivalent.

(i)\ $D(R)\ $is a Noetherian space.

(ii)\ $R\ $is a field.
\end{theorem}

\begin{proof}
$(i)\Rightarrow(ii):\ $Let $D(R)\ $be a Noetherian space. Now, we will show
that every nonzero element is invertible. Choose a nonzero element $a\in R$
and assume that $a$ is not invertible. It is clear that
\[
U_{a}\subseteq U_{a^{2}}\subseteq\cdots\subseteq U_{a^{n}}\subseteq
U_{a^{n+1}}\subseteq\cdots
\]
and $\{U_{a^{n}}\}_{n\in%
\mathbb{N}
}$ is a subfamily of the basis $\{U_{x}\}_{x\in R^{\#}}$ of $D(R).\ $Then by
Theorem \ref{tNotherian}, $\{U_{a^{n}}\}_{n\in%
\mathbb{N}
}$ has a maximal element. Let $U_{a^{k}}$ be a maximal element of $\{U_{a^{n}%
}\}_{n\in%
\mathbb{N}
}.\ $Thus $U_{a^{k}}\subseteq U_{a^{k+1}}$ gives that $U_{a^{k}}=U_{a^{k+1}%
}.\ $This implies that $a^{k+1}|a^{k}\ $and we can write $a^{k}=a^{k+1}y$ for
some $y\in R.\ $Since $R\ $is an integral domain, we have $1=ay$\ which is a
contradiction. Thus, $a$ is invertible and $R\ $is a field.

$(ii)\Rightarrow(i):\ $Let $R\ $be a field. Then by Remark \ref{remark},
$D(R)\ $is an empty space which is trivially Noetherian.
\end{proof}

\begin{corollary}
Assume that $R$ is an integral domain which is not a field. Then $D(R)$ can not be a
Noetherian topological space.
\end{corollary}

\section{Infinitude of primes}

Euclid first proved the infinitude of prime numbers about 300 BC. Nearly 2000
years after Euclid's proof, the German mathematician Christian Goldbach gave a
new proof of the infinitude of primes using Fermat's numbers. Afterward, in
1737, Euler used harmonic series and gave another proof of the infinitude of
primes. The next well-known proof of the infinitude of primes is due to
another German mathematician Peter Dirichlet. In 1837, he published a theorem
on arithmetic progression. His proof combined the concepts in mathematical
analysis and number theory giving rise to a new field of mathematics known as
analytic number theory. Afterward, the Hungarian mathematician, Paul Erd\"{o}s
proved the infinitude of primes by determining a lower bound for $\pi
(n),\ $where $\pi(n)$ denotes the number of primes less than or equal to
$n.\ $He showed that$\ \pi(n)\geq\frac{1}{2}\log_{2}n$ and thus he obtained
the infinitude of primes. In 1955, Harry Furstenberg defined a topology on the
integers $%
\mathbb{Z}
$ by using arithmetic progression. He gave the first topological proof of the
infinitude of primes \cite{Furs55}. For many other proofs about the infinitude
of primes, we refer \cite{Harison} and \cite{Ozcan} to the reader. In this
section, we will give an alternative proof of the infinitude of primes in
terms of the divisor topology $D(R)$. In the following theorem, we denote the
cardinality of a set $S\ $by $\left\vert S\right\vert \ $and the set of units
in $R$ by $u(R).$

\begin{theorem}
\label{infty}(\textbf{Infinitude of primes) }Let $R$ be a UFD which is not a
field. If $\left\vert u(R)\right\vert <\left\vert R\right\vert ,\ $then $R$
has infinitely many prime elements.
\end{theorem}

\begin{proof}
Assume that $R$ has finitely many prime elements, say $a_{1},a_{2}%
,\ldots,a_{n}.\ $Since $R\ $is a UFD, every nonzero nonunit element can be
written of the form $a_{1}^{k_{1}}a_{2}^{k_{2}}\cdots a_{n}^{k_{n}}$ where
$k_{1},k_{2},\ldots,k_{n}\in%
\mathbb{N}
\cup\{0\}.$\ Also, since $R$ is integral domain, we have $\left\vert
u(R)\right\vert \leq\left\vert R^{\#}\right\vert $. Thus $R$ has countable
element. If $\left\vert u(R)\right\vert $ is infinite, then $\left\vert
u(R)\right\vert $ has at least countable element, and so $\left\vert
u(R)\right\vert =\left\vert R\right\vert $ which is a contradiction. Thus we
have $\left\vert u(R)\right\vert $ is finite. By Theorem \ref{tdense},
$EC(Irr(R))=\{a_{1},a_{2},\ldots,a_{n}\}$ is dense in $D(R).\ $Thus we
conclude that $\overline{EC(Irr(R))}=%
{\displaystyle\bigcup\limits_{i=1}^{n}}
\overline{\left\{  \left[  a_{i}\right]  \right\}  }=EC(R^{\#})$ which implies
that $%
{\displaystyle\bigcap\limits_{i=1}^{n}}
\overline{\left\{  \left[  a_{i}\right]  \right\}  }^{C}=\emptyset.\ $Now, put
$x_{m}=a_{1}^{m}+a_{2}a_{3}\cdots a_{n}$.\ Since $\left\vert u(R)\right\vert $
is finite, $\{x_{m}:m\in%
\mathbb{N}
\}$ contains a nonunit element, say $x_{t}=a_{1}^{t}+a_{2}a_{3}\cdots a_{n}$.
On the other hand, by Proposition \ref{pclosure}, we have $\overline{\left\{
\left[  a_{i}\right]  \right\}  }^{C}=\{[b]\in EC(R^{\#}):a_{i}$ does not
divide $b\}.$ This gives $x_{t}=a_{1}^{t}+a_{2}a_{3}\cdots a_{n}\in%
{\displaystyle\bigcap\limits_{i=1}^{n}}
\overline{\left\{  \left[  a_{i}\right]  \right\}  }^{C}$ which is a
contradiction. Hence, $R$ has infinitely many prime elements.
\end{proof}

As an immediate consequence of the previous theorem, we have the following result.

\begin{corollary}
The ring $%
\mathbb{Z}
$ of integers has infinitely many primes.
\end{corollary}

\begin{acknowledgement} On behalf of all authors, the corresponding author states that there is no conflict of interest.
\end{acknowledgement}

\end{document}